\documentclass[12pt, twoside]{article}
\usepackage{amsfonts}
\usepackage{amsmath, amsthm, amscd, amsfonts, xpatch}
\usepackage{times}

\usepackage{enumitem,kantlipsum}
\usepackage[english]{babel}
\usepackage[T1]{fontenc}
\usepackage[utf8]{inputenc}
\usepackage{bm}
\usepackage{ifxetex}
\usepackage{pb-diagram}
\usepackage{mathtools}
\usepackage{mathrsfs}
\usepackage{latexsym}
\usepackage{amssymb}
\usepackage{caption}

\usepackage{graphicx,float}
\usepackage{tikz}
\usepackage{pst-node}
\usepackage{tikz-cd}

\def\titlerunning#1{\gdef\titrun{#1}}
\makeatletter
\def\author#1{\gdef\autrun{\def\and{\unskip, }#1}\gdef\@author{#1}}
\def\address#1{{\def\and{\\\hspace*{18pt}}\renewcommand{\thefootnote}{}%
\footnote {#1}}%
\markboth{\autrun}{\titrun}}
\makeatother
\def\email#1{e-mail: #1}

\def\keywords#1{\par\medskip
\noindent\textbf{Keywords.} #1}
\def\MSC#1{\par\medskip
\noindent\textbf{MSC.} #1}
\usepackage{lineno}
\usepackage{csquotes}

\usepackage[a4paper,inner=2cm,outer=2cm,top=2cm,bottom=2cm,footskip=1cm]{geometry}

\definecolor{persianblue}{rgb}{0.11, 0.22, 0.73}
\definecolor{persiangreen}{rgb}{0.0, 0.65, 0.58}

\usepackage{url}

\usepackage{nameref}
\usepackage[hidelinks]{hyperref}
\hypersetup{
  colorlinks   = true, 
  urlcolor     = persiangreen, 
  linkcolor    = persianblue, 
  citecolor   = persianblue 
}
\usepackage[capitalize]{cleveref}
\Crefname{theorem}{Theorem}{Theorems}
\Crefname{lemma}{Lemma}{Lemmas}
\crefname{claim}{Claim}{Claims}

\usepackage[defernumbers=true,maxbibnames=6,backend=biber,style=numeric,giveninits=true,sortcites=true,doi=false,isbn=false,url=false,safeinputenc]{biblatex}
\usepackage[nottoc,numbib]{tocbibind} 

\newtheorem{theorem}{Theorem}[section]
\newtheorem*{theorem*}{Theorem}
\newtheorem*{mainthm*}{Main Theorem}

\newtheorem{lemma}[theorem]{Lemma}

\newtheorem{proposition}[theorem]{Proposition}
\newtheorem{corollary}[theorem]{Corollary}

\newtheorem{problem*}{Problem}[]

\newtheorem{fact}[theorem]{Fact}
\newtheorem{definition}[theorem]{Definition}

\numberwithin{equation}{section}

\newtheorem{imp-remark}[theorem]{\textbf{Important Remark}}

\numberwithin{equation}{section}

\theoremstyle{remark}

\makeatletter
\let\qed@empty\openbox 
\def\@begintheorem#1#2[#3]{%
  \deferred@thm@head{%
    \the\thm@headfont\thm@indent
    \@ifempty{#1}
      {\let\thmname\@gobble}
      {\let\thmname\@iden}%
    \@ifempty{#2}
      {\let\thmnumber\@gobble\global\let\qed@current\qed@empty}
      {\let\thmnumber\@iden\xdef\qed@current{#2}}%
    \@ifempty{#3}
      {\let\thmnote\@gobble}
      {\let\thmnote\@iden}%
    \thm@swap\swappedhead
    \thmhead{#1}{#2}{#3}%
    \the\thm@headpunct\thmheadnl\hskip\thm@headsep
  }\ignorespaces
}
\renewcommand{\qedsymbol}{%
  \ifx\qed@thiscurrent\qed@empty
    \qed@empty
  \else
    \fbox{\scriptsize\qed@thiscurrent}%
  \fi
}
\renewcommand{\proofname}{%
  Proof%
  \ifx\qed@thiscurrent\qed@empty
  \else
    \ of \qed@thiscurrent
  \fi
}
\xpretocmd{\proof}{\let\qed@thiscurrent\qed@current}{}{}
\newenvironment{proof*}[1]
  {\def\qed@thiscurrent{\ref{#1}}\proof}
  {\endproof}
\makeatother

\makeatletter
\let\blx@rerun@biber\relax
\makeatother

\begin{document}


\baselineskip=17pt



\titlerunning{IGMP And The Approximation Property}
\title{Indestructible Guessing Models And The Approximation Property}

\author{Rahman Mohammadpour}

\date{}

\maketitle

\address{Rahman Mohammadpour: Institut für Diskrete Mathematik und Geometrie, TU Wien,
1040 Vienna, Austria.\\ \email{rahmanmohammadpour@gmail.com}}



\begin{abstract} 

In this short note, we shall prove some observations regarding the connection between indestructible $\omega_1$-guessing models and the $\omega_1$-approximation property of forcing notions.

\MSC{03E35}

\end{abstract}

\section{Introduction and Basics}
Viale and Wei\ss~\cite{VW2011} introduced and used the notion of an $\omega_1$-guessing model to reformulate the principle ${\rm ISP}(\omega_2)$ and to show, among other things, that
${\rm ISP}(\omega_2)$ follows form $\rm PFA$.
Cox and Krueger \cite{CK2017} introduced and studied  indestructible $\omega_1$-guessing sets of size $\omega_1$, i.e., the $\omega_1$-guessing sets which remains valid in  generic extensions by any $\omega_1$-preserving forcing. They formulated an analogous principle, denoted by ${\rm IGMP}(\omega_2)$, and showed that it follows from $\rm PFA$. Among other things, they  showed that ${\rm IGMP}(\omega_2)$ implies
the Suslin Hypothesis. More generally, they proved that
under $\rm IGMP(\omega_2)$, if $(T,<_T)$ is a nontrivial tree of height and size $\omega_1$, then the forcing notion $(T,\geq_T)$ collapses  $\omega_1$.
This theorem  establishes  a connection between indestructible $\omega_1$-guessing sets and the $\omega_1$-approximation property of forcing notions. 
In this short paper, we examine a close inspection of the connection between the indestructibility of $\omega_1$-guessing models and the $\omega_1$-approximation property of forcing notions. In particular, we shall show that under ${\rm GMP}(\omega_2)$, if $\mathbb P$ is an $\omega_1$-preserving forcing which is proper for  $\omega_1$-guessing models of size $\omega_1$, then $\mathbb P$ has the $\omega_1$-approximation property if and only if the guessing models are indestructible by $\mathbb P$.

\subsection*{Guessing models}

Throughout this paper, by the stationarity of a set $\mathcal S\subseteq\mathcal P_{\omega_2}(H_\theta)$, we shall mean that for every function $F:\mathcal P_{\omega}(H_\theta)\rightarrow \mathcal P_{\omega_2}(H_\theta)$, there  is $M\prec H_\theta$ in $S$ with $M\cap\omega_2\in\omega_2$ such that $M$ is closed under $F$.
We say a set $x$ is \emph{bounded} in a set or class $M$ if there exists $X\in M$ with $x\subseteq X$.

\begin{definition}[Viale-Wei\ss~\cite{VW2011}]
A set $M$  is called \textbf{$\bm{\omega_1}$-guessing}  if and only if the following are equivalent for every $x$ which is bounded in $M$.
\begin{enumerate}
\item $x$ is  $\omega_1$-approximated in $M$, i.e., for every countable $a\in M$, $a\cap x\in M$.
\item  $x$ is guessed in $M$, i.e., there  exists $x^*\in M$ with $x^*\cap M=x\cap M$. .

\end{enumerate}

\end{definition}

\begin{definition}[$\rm GMP(\omega_2)$]
$\rm GMP(\omega_2)$ states that
for every regular $\theta\geq\omega_2$, the set of $\omega_1$-guessing elementary submodels  of $H_\theta$ of size $\omega_1$ is stationary in $\mathcal P_{\omega_2}(H_\theta)$.
\end{definition}

\begin{definition}[Cox--Krueger \cite{CK2017}]\leavevmode
\begin{enumerate}

\item An $\omega_1$-guessing set  is said to be \textbf{indestructibly $\bm \omega_1$-guessing}  if it remain $\omega_1$-guessing  in any $\omega_1$-preserving forcing extension.

\item Let ${\rm IGMP}(\omega_2)$ state that for every  regular cardinal $\theta\geq\omega_2$, there exist stationarily many $M\in\mathcal P_{\omega_2}(H_\theta)$ such that $M$ is indestructibly $\omega_1$-guessing.
\end{enumerate}
\end{definition}

We shall use the following without mentioning.
\begin{fact}
Let $\theta\geq\omega_2$ be a cardinal.
Assume $M\prec H_\theta$ is $\omega_1$-guessing. Then $\omega_1\subseteq M$.
\end{fact}
\begin{proof}
See \cite[Lemma 2.3]{CK2017}
\end{proof}

\subsection*{Generalised Proper Forcing}
Let $\mathbb P$ be a forcing. Assume that $M\prec H_\theta$ with $\mathbb P,\mathcal P(\mathbb P)\in M$.
 A condition $p\in\mathbb P$ is $(M,\mathbb P)$-generic, if for every dense set $D\subseteq\mathbb P$ which belongs to $M$, $M\cap D$ is pre-dense below $p$.
The proof of the following is standard.

  \begin{lemma}\label{gen-lemma}
  Suppose that $\mathbb P$ is a forcing. Assume that $M\prec H_\theta$ with $\mathbb P,\mathcal P(\mathbb P)\in M$. Let $p\in\mathbb P$. Then $p$ is $(M,\mathbb P)$-generic if and only if  $p\Vdash ``M[\dot{G}]\cap H_\theta^V=M"$.
\end{lemma}
  \begin{proof}[\nopunct]  
  \end{proof}

  Let $\theta$ be a sufficiently large regular cardinal.
  A forcing $\mathbb P$ is said to be \textbf{proper for} $\bm{ \mathcal S}$, where $\mathcal S\subseteq\mathcal P_{\omega_2}(H_\theta)$ consists of elementary submodels of $(H_\theta,\in,\mathbb P)$,   if  for every $M\in \mathcal S$ and every $p\in M\cap\mathbb P$, there is an $(M,\mathbb P)$-generic condition $q\leq p$. A forcing is said to be \textbf{proper for models of size} $\bm{\omega_1}$, if for every sufficiently large regular cardinal $\theta$, $\mathbb P$ is proper for 
$
\{M\prec (H_\theta,\in,\mathbb P): \omega_1\subseteq M \mbox{ and } |M|=\omega_1\}
$.
 It is  easy to see that every forcing which is proper for a stationary set $\mathcal S\subseteq\mathcal P_{\omega_2}(H_\theta)$  preserves $\omega_2$.

  \begin{lemma}\label{proper-lemma}
  Suppose that $\mathbb P$ is  proper for a stationary set $\mathcal S\subseteq\mathcal P_{\omega_2}(H_\theta)$. Then $\mathbb P$ preserves the stationarity of $\mathcal S$. 
  \end{lemma}
\begin{proof}
Assume that $p\in \mathbb P$ forces that $ ``\dot{F}:\mathcal P_{\omega}(H_\theta^V)\rightarrow \mathcal P_{\omega_2}(H_\theta^V) \mbox{ is a function}"$. 
Pick a sufficiently large regular cardinal $\theta^*>\theta$ with $\dot{F}\in H_{\theta^*}$.
Pick $M^*\prec H_{\theta^*}$ with $\omega_1\cup\{H_\theta,\dot{F},p\}\subseteq M^*$ and $M\coloneqq M^*\cap H_\theta\in \mathcal S$. Such a model exists by our assumption on the stationarity of $\mathcal S$.
Since $\mathbb P$ is proper for $\mathcal S$, we can extend
$p$ to  an $(M,\mathbb P)$-generic condition $q$.
Assume that $G\subseteq\mathbb P$ is a $V$-generic filter with $q\in G$.
Now in $V[G]$, $M[G]$ is closed under $F$, as $\omega_1\subseteq M$. By \cref{gen-lemma}, $M[G]\cap H_\theta^V=M$, and hence $M$ is closed under $F$.
Thus $q$ forces that $\check{M}$ is closed under $\dot{F}$. Since $p$ was arbitrary, the maximal condition forces that $\mathcal S$ is stationary.

\end{proof}

Let us recall the definition of  the $\omega_1$-approximation property of a forcing notion.
\begin{definition}[Hamkins \cite{Hamkins2001}]
A forcing notion $\mathbb P$ has the 
$\omega_1$-approximation property in $V$ if for every $V$-generic filter $G\subseteq\mathbb P$, and for every $x\in V[G]$ which is bounded in $V$ so that for every countable $a\in V$, $a\cap x\in V$, then $x\in V$.
\end{definition}

\section{IGMP and the Approximation Property}\label{secmain}

\begin{lemma}\label{Lemma0}
Suppose that  $\mathbb P$ has the $\omega_1$-approximation property. Assume that $M\prec H_\theta$ is $\omega_1$-guessing, for some $\theta\geq\omega_2$. Then $\mathbb P$ forces  $M$ to  be $\omega_1$-guessing.
\end{lemma}
\begin{proof}
Let $G\subseteq\mathbb P$ be a $V$-generic filter.
Fix  $x\in V[G]$ and assume that $x\subseteq X\in M$ is
$\omega_1$-approximated in $M$. We claim that $x\cap M$ is $\omega_1$-approximated in $V$, which in turn implies that $x\cap M\in V$. Then, since $M$ is $\omega_1$-guessing in $V$, $x$ is guessed in $M$. 
To see that $x\cap M$ is $\omega_1$-approximated in $V$, fix a countable set $a\in V$. By  \cite[Theorem 1.4]{Krueger2020},
 there is a countable set $b\in M$ with $a\cap M\cap X\subseteq b$. Thus
 $a\cap x\cap M=a\cap x\cap b\in V$, since $a\in V$ and $x\cap b\in M\subseteq V$.

\end{proof}

\begin{definition}
For an $\omega_1$-preserving forcing notion $\mathbb P$, we let  $\mathbb P{-\rm IGMP}(\omega_2)$   states that for every sufficiently large regular $\theta$, the set of $\omega_1$-guessing sets of size $\omega_1$ which remain $\omega_1$-guessing after forcing with $\mathbb P$, is stationary in $\mathcal P_{\omega_2}(H_\theta)$.
\end{definition}
 It is clear that ${\rm IGMP}(\omega_2)$ implies that  $\mathbb P-{\rm IGMP}(\omega_2)$ holds, for all $\omega_1$-preserving forcing $\mathbb P$. Note that ${\rm IGMP}(\omega_2)$ is a diagonal version of the statement that,  for every $\omega_1$-preserving forcing $\mathbb P$, $\mathbb P-{\rm IGMP}(\omega_2)$ holds. It is also worth mentioning that the ${\rm IGMP}(\omega_2)$ obtained by Cox and Kruger has the property that every indestructible $\omega_1$-guessing model remains $\omega_1$-guessing in any outer transitive extension with the same $\omega_1$.

\begin{proposition}\label{Lemma1}
   Assume that $\mathbb P$ is an $\omega_1$-preserving forcing. Suppose that for every sufficiently large regular cardinal $\theta$, $\mathbb P$ is proper for a stationary set $\mathfrak{G}_\theta\subseteq\mathcal P_{\omega_2}(H_\theta)$ of $\omega_1$-guessing elementary submodels of $H_\theta$. Then the following are equivalent.
\begin{enumerate}
\item $\mathbb P$ has the $\omega_1$-approximation property.
\item  Every $\omega_1$-guessing model is indestructible by $\mathbb P$.

\end{enumerate}

\end{proposition}
\begin{proof}
Observe that the implication $1.\Rightarrow 2.$
 follows from \cref{Lemma0}. To see that the implication $2.\Rightarrow 1.$ holds true,
 fix  an $\omega_1$-preserving forcing $\mathbb P$ and assume that the maximal condition of $\mathbb P$ forces $\dot A$ is a countably approximated subset of an ordinal $\gamma$. Pick a regular $\theta$, with $\gamma,\dot{A},\mathcal P(\mathbb P)\in H_\theta$.
   Assume that $\mathfrak G\coloneqq\mathfrak{G}_\theta\subseteq\mathcal P_{\omega_2}(H_\theta)$ is a stationary set of $\omega_1$-guessing elementary submodels of $H_\theta$  for which $\mathbb P$ is proper.
  We shall show that $\mathbb P\Vdash ``\dot{A}\in V"$. Let $G\subseteq\mathbb P$ be a $V$-generic filter, and
 set 
 \[
 \mathcal S\coloneqq\{M\in\mathfrak G: p,\gamma,\dot A,\mathbb P\in M \mbox{ and } M[G]\cap H_\theta^V=M \}.
 \]
In $V[G]$, $\mathcal S$ is stationary in $\mathcal P_{\omega_2}(H_\theta^V)$. To see this, let $F:\mathcal P_{\omega}(H_\theta^V)\rightarrow \mathcal P_{\omega_2}(H_\theta^V)$ be defined by $F(x)=\{\dot{y}^G\}$ if $x=\{\dot{y}\}$ for some $\mathbb P$-name $\dot{y}$  with $\dot{y^G}\in H_\theta^V$, and otherwise let $F(x)=\{p,\gamma,\dot{A},\mathbb P\}$. By \cref{proper-lemma}, the set of models in $\mathfrak G$ which are closed under $F$ is stationary. Observe that
a model $M\in\mathfrak{G}$ is closed under $F$ if and only if $M\in\mathcal S$.

Let  $A=\dot{A}^G$ and fix $M\in\mathcal S$.
We claim that $A$ is countably approximated in $M$. Let $a\in M$ be a countable subset of $\gamma$. 
Let $D_a$ be the set of conditions deciding $\dot A \cap a$. Then $D_a$ belongs to $M$ and is dense in $\mathbb P$, as the maximal condition forces that $\dot{A}$ is countably approximated in $V$. By the elementarity of $M[G]$ in $H_\theta[G]$,
there is $p\in G\cap D_a\cap M[G]$. But then $p\in M$, as $D_a\in H_\theta^V$.
Working in $V$,  the elementarity of $M$ in $H_\theta$ implies that there is some $b\in M$ such that, $p\Vdash`` \check b=\dot A\cap a"$. Since $p\in G$, we have 
 $A\cap a=b\in M$. Thus $A$ is countably approximated in $M$.  By our assumption, $M$ is an $\omega_1$-guessing set in $V[G]$. Thus there is $A^*$ in $M$, and hence in $V$,
such that $A^*\cap M=A\cap M$. 

 Working  in $V[G]$ again, for every $M\in\mathcal S$, there is, by the previous paragraph, a set $A^*_M\in M$ such that $A^*_M\cap M=A\cap M$.
This defines a regressive function $M\mapsto A^*_M$  on $\mathcal S$. As $\mathcal S$ is stationary in $H^V_\theta$, there are a set $A^*\in H^V_\theta$ and
 a stationary set $\mathcal S^*\subseteq\mathcal S$ such that for every $M\in\mathcal S^*$, we have
  $A^*\cap M=A\cap M$. Since $A\subseteq\bigcup\mathcal S^*$, we have $A^* =A$, which in turn implies that $A\in V$.

\end{proof}

\begin{corollary}\label{cor1.2}
Assume $\rm GMP(\omega_2)$. Suppose that $\mathbb P$ is an $\omega_1$-preserving forcing which is also proper for models of size $\omega_1$.  Then the following are equivalent.
\begin{enumerate}
\item  $\mathbb P{-\rm IGMP}(\omega_2)$ holds.
\item $\mathbb P$ has the $\omega_1$-approximation property.
\end{enumerate}

\end{corollary}
\begin{proof}[\nopunct]
\end{proof}

Note that if $(T,<_T)$ is a tree of height and size $\omega_1$, then $(T,\geq_T)$ is proper for models of size $\omega_1$. However, it does not have the $\omega_1$-approximation property if it is nontrivial as a forcing notion. We have the following generalization of  \cite[Theorem 3.7]{CK2017}.
\begin{theorem}\label{theorem-gen}
Assume  ${\rm IGMP}(\omega_2)$. Then every $\omega_1$-preserving forcing which is proper for models of size $\omega_1$ has the $\omega_1$-approximation property.
In particular, under ${\rm IGMP}(\omega_2)$ every $\omega_1$-preserving forcing of size $\omega_1$ has the $\omega_1$-approximation property.
\end{theorem}
\begin{proof}
Let $\mathbb P$ be an $\omega_1$-preserving function which is proper for models of size $\omega_1$. As
 ${\rm IGMP}(\omega_2)$ holds,   \cref{Lemma1} implies that $\mathbb P$ has the $\omega_1$-approximation property.
\end{proof}

For a class  $\mathfrak K$ of forcing notions, we let 
 ${\rm FA}(\mathfrak K,\omega_1)$ state that for every $\mathbb P\in\mathfrak K$, and every $\omega_1$-sized family $\mathcal D$ of dense subsets of $\mathbb P$, there is a $\mathcal D$-generic filter $G\subseteq\mathbb P$.

\begin{lemma}\label{Lemma2}
Assume ${\rm FA}(\{\mathbb P\},\omega_1)$, for some
forcing notion $\mathbb P$. Suppose that  $M$ is an $\omega_1$-guessing set of size $\omega_1$.
Then $\mathbb P$ forces that $M$ is $\omega_1$-guessing.
\end{lemma}
\begin{proof}
Assume towards a contraction that for some $p_0\in\mathbb P$, some ordinal $\delta\in M$, and some $\mathbb P$-name $\dot{A}$, $p_0$ forces that $\dot{A}\subseteq\delta$  is countably approximated in $M$, but is not guessed in $M$. We may assume that $p_0$ is the maximal condition of $\mathbb P$.

\begin{itemize}
\item For every $\alpha\in M\cap\delta$, let $D_\alpha\coloneqq\{p\in\mathbb P: p  \mbox{ decides } \alpha\in \dot{A}\}$. 
\item For every $x\in M\cap \mathcal P_{\omega_1}(\delta)$, let $E_x\coloneqq\{p\in\mathbb P:~  \exists y\in M~  p\Vdash ``\dot{A}\cap x=\check{y}"\}$.
\item For every $B\in M\cap \mathcal P(\delta)$, let $F_B\coloneqq\{p\in\mathbb P: \exists\xi\in M, (p\Vdash``\xi\in\dot{A}" )\Leftrightarrow \xi\notin B\}$.
\end{itemize}
By our assumptions, it is easily seen that the above sets are dense in $\mathbb P$.
Let \[
\mathcal D=\{D_\alpha,E_x,F_B: \alpha,x,B \mbox{ as above }\}.
\]
We have $|\mathcal D|=\omega_1$.
By ${\rm FA}(\{\mathbb P\},\omega_1)$, there is a $\mathcal D$-generic filter $G\subseteq\mathbb P$.
Let $A^*\subseteq\delta$ be defined by 
\[
\alpha\in A^* \mbox{ if and onlty if }~ \exists p\in G \mbox{ with } p\Vdash ``\alpha\in\dot{A}."
\]
By the $\mathcal D$-genericity of $G$, $A^*$ is a well-defined subset of $\delta$ which is countably approximated in $M$ but not guessed in $M$, a contradiction!

\end{proof}

The following theorem is  immediate from \cref{cor1.2,Lemma2}.

\begin{theorem}
Let $\mathfrak K$ be a class of forcings which are proper for models of size $\omega_1$.
Assume that  ${\rm FA}(\mathfrak K,\omega_1)$  and ${\rm GMP}(\omega_2)$ hold. Then, for every forcing $\mathbb P\in\mathfrak K$, $\mathbb P{\rm-IGMP}(\omega_2)$ holds, and $\mathbb P$ has the $\omega_1$-approximation property.
\end{theorem}
\begin{proof}[\nopunct]

\end{proof}


\footnotesize
\noindent\textit{Acknowledgements.}
The author's research was supported through the project M 3024 by the Austrian Science Fund (FWF). This work was partly conducted when the author was a PhD student at the University of Paris. The author would like to thank his former supervisor Boban Veličković for his support and  Mohammad Golshani for his careful reading of an earlier draft of this paper.

\normalsize
\baselineskip=17pt

\printbibliography[heading=bibintoc]

\end{document}